\documentclass[11pt]{amsart}

\usepackage{graphics}
\usepackage{color}
\usepackage[a4paper, margin=3cm]{geometry}

\usepackage{amssymb,enumerate}
\usepackage{amsmath}
\usepackage{bbm}           
\usepackage{bm}
\usepackage{eso-pic,graphicx}
\usepackage{tikz}
\usepackage{cite}
\usepackage{esint}
\usepackage[colorlinks=true, pdfstartview=FitV, linkcolor=blue, citecolor=blue, urlcolor=blue]{hyperref}

\usepackage{graphicx}
\usepackage{pslatex}
\usepackage{amsmath,amsfonts}
\usepackage{rotating}
\usepackage{amssymb}
\usepackage{verbatim}
\usepackage{rotating}
\usepackage{mathrsfs}

\makeatletter
\def\Ddots{\mathinner{\mkern1mu\raise\p@
\vbox{\kern7\p@\hbox{.}}\mkern2mu
\raise4\p@\hbox{.}\mkern2mu\raise7\p@\hbox{.}\mkern1mu}}
\makeatother

\def\XXint#1#2#3{{\setbox0=\hbox{$#1{#2#3}{\int}$}
\vcenter{\hbox{$#2#3$}}\kern-.5\wd0}}

\begin{document}
\newtheorem{theorem}{Theorem}
\newtheorem{proposition}[theorem]{Proposition}
\newtheorem{conjecture}[theorem]{Conjecture}
\def\theconjecture{\unskip}
\newtheorem{corollary}[theorem]{Corollary}
\newtheorem{lemma}[theorem]{Lemma}
\newtheorem{claim}[theorem]{Claim}
\newtheorem{sublemma}[theorem]{Sublemma}
\newtheorem{observation}[theorem]{Observation}
\theoremstyle{definition}
\newtheorem{definition}{Definition}
\newtheorem{notation}[definition]{Notation}
\newtheorem{remark}[definition]{Remark}
\newtheorem{question}[definition]{Question}
\newtheorem{questions}[definition]{Questions}
\newtheorem{example}[definition]{Example}
\newtheorem{problem}[definition]{Problem}
\newtheorem{exercise}[definition]{Exercise}
 \newtheorem{thm}{Theorem}
 \newtheorem{cor}[thm]{Corollary}
 \newtheorem{lem}{Lemma}[section]
 \newtheorem{prop}[thm]{Proposition}
 \theoremstyle{definition}
 \newtheorem{dfn}[thm]{Definition}
 \theoremstyle{remark}
 \newtheorem{rem}{Remark}
 \newtheorem{ex}{Example}
 \numberwithin{equation}{section}
\def\C{\mathbb{C}}
\def\R{\mathbb{R}}
\def\Rn{{\mathbb{R}^n}}
\def\Rns{{\mathbb{R}^{n+1}}}
\def\Sn{{{S}^{n-1}}}
\def\M{\mathbb{M}}
\def\N{\mathbb{N}}
\def\Q{{\mathbb{Q}}}
\def\Z{\mathbb{Z}}
\def\F{\mathcal{F}}
\def\L{\mathcal{L}}
\def\S{\mathcal{S}}
\def\supp{\operatorname{supp}}
\def\essi{\operatornamewithlimits{ess\,inf}}
\def\esss{\operatornamewithlimits{ess\,sup}}

\numberwithin{equation}{section}
\numberwithin{thm}{section}
\numberwithin{theorem}{section}
\numberwithin{definition}{section}
\numberwithin{equation}{section}

\def\earrow{{\mathbf e}}
\def\rarrow{{\mathbf r}}
\def\uarrow{{\mathbf u}}
\def\varrow{{\mathbf V}}
\def\tpar{T_{\rm par}}
\def\apar{A_{\rm par}}

\def\reals{{\mathbb R}}
\def\torus{{\mathbb T}}
\def\scriptm{{\mathcal T}}
\def\heis{{\mathbb H}}
\def\integers{{\mathbb Z}}
\def\z{{\mathbb Z}}
\def\naturals{{\mathbb N}}
\def\complex{{\mathbb C}\/}
\def\distance{\operatorname{distance}\,}
\def\support{\operatorname{support}\,}
\def\dist{\operatorname{dist}\,}
\def\Span{\operatorname{span}\,}
\def\degree{\operatorname{degree}\,}
\def\kernel{\operatorname{kernel}\,}
\def\dim{\operatorname{dim}\,}
\def\codim{\operatorname{codim}}
\def\trace{\operatorname{trace\,}}
\def\Span{\operatorname{span}\,}
\def\dimension{\operatorname{dimension}\,}
\def\codimension{\operatorname{codimension}\,}
\def\nullspace{\scriptk}
\def\kernel{\operatorname{Ker}}
\def\ZZ{ {\mathbb Z} }
\def\p{\partial}
\def\rp{{ ^{-1} }}
\def\Re{\operatorname{Re\,} }
\def\Im{\operatorname{Im\,} }
\def\ov{\overline}
\def\eps{\varepsilon}
\def\lt{L^2}
\def\diver{\operatorname{div}}
\def\curl{\operatorname{curl}}
\def\etta{\eta}
\newcommand{\norm}[1]{ \|  #1 \|}
\def\expect{\mathbb E}
\def\bull{$\bullet$\ }

\def\blue{\color{blue}}
\def\red{\color{red}}

\def\xone{x_1}
\def\xtwo{x_2}
\def\xq{x_2+x_1^2}
\newcommand{\abr}[1]{ \langle  #1 \rangle}

\newcommand{\Norm}[1]{ \left\|  #1 \right\| }
\newcommand{\set}[1]{ \left\{ #1 \right\} }
\newcommand{\ifou}{\raisebox{-1ex}{$\check{}$}}
\def\one{\mathbf 1}
\def\whole{\mathbf V}
\newcommand{\modulo}[2]{[#1]_{#2}}
\def \essinf{\mathop{\rm essinf}}
\def\scriptf{{\mathcal F}}
\def\scriptg{{\mathcal G}}
\def\scriptm{{\mathcal M}}
\def\scriptb{{\mathcal B}}
\def\scriptc{{\mathcal C}}
\def\scriptt{{\mathcal T}}
\def\scripti{{\mathcal I}}
\def\scripte{{\mathcal E}}
\def\scriptv{{\mathcal V}}
\def\scriptw{{\mathcal W}}
\def\scriptu{{\mathcal U}}
\def\scriptS{{\mathcal S}}
\def\scripta{{\mathcal A}}
\def\scriptr{{\mathcal R}}
\def\scripto{{\mathcal O}}
\def\scripth{{\mathcal H}}
\def\scriptd{{\mathcal D}}
\def\scriptl{{\mathcal L}}
\def\scriptn{{\mathcal N}}
\def\scriptp{{\mathcal P}}
\def\scriptk{{\mathcal K}}
\def\frakv{{\mathfrak V}}
\def\C{\mathbb{C}}
\def\D{\mathcal{D}}
\def\R{\mathbb{R}}
\def\Rn{{\mathbb{R}^n}}
\def\rn{{\mathbb{R}^n}}
\def\Rm{{\mathbb{R}^{2n}}}
\def\r2n{{\mathbb{R}^{2n}}}
\def\Sn{{{S}^{n-1}}}
\def\M{\mathbb{M}}
\def\N{\mathbb{N}}
\def\Q{{\mathcal{Q}}}
\def\Z{\mathbb{Z}}
\def\F{\mathcal{F}}
\def\L{\mathcal{L}}
\def\G{\mathscr{G}}
\def\ch{\operatorname{ch}}
\def\supp{\operatorname{supp}}
\def\dist{\operatorname{dist}}
\def\essi{\operatornamewithlimits{ess\,inf}}
\def\esss{\operatornamewithlimits{ess\,sup}}
\def\dis{\displaystyle}
\def\dsum{\displaystyle\sum}
\def\dint{\displaystyle\int}
\def\dfrac{\displaystyle\frac}
\def\dsup{\displaystyle\sup}
\def\dlim{\displaystyle\lim}
\def\bom{\Omega}
\def\om{\omega}
\def\BMO{\rm BMO}
\def\CMO{\rm CMO}

\author[S. Wang]{Shifen Wang}
\address{Shifen Wang\\
	School of Mathematical Sciences \\
	Beijing Normal University \\
	Laboratory of Mathematics and Complex Systems \\
	Ministry of Education \\
	Beijing 100875 \\
	People's Republic of China}
\email{wsfrong@mail.bnu.edu.cn}

\author[Q. Xue]{Qingying Xue$^{*}$}
\address{Qingying Xue\\
	School of Mathematical Sciences \\
	Beijing Normal University \\
	Laboratory of Mathematics and Complex Systems \\
	Ministry of Education \\
	Beijing 100875 \\
	People's Republic of China}
\email{qyxue@bnu.edu.cn}


\thanks{ The second author was supported partly by NNSF of China (Nos. 11671039, 11871101) and NSFC-DFG (No. 11761131002).
	\thanks{$^{*}$ Corresponding author, e-mail address: qyxue@bnu.edu.cn}}

\date{\today}
\title[ on weighted Compactness of Commutators ]
{\bf On weighted Compactness of commutators of bilinear maximal Calder\'on-Zygmund singular integral operators}

\begin{abstract} 
Let $T$ be a bilinear Calder\'on-Zygmund singular integral operator and $T^*$ be its corresponding truncated maximal operator. For any $b\in\BMO(\rn)$ and $\vec{b}=(b_1,\ b_2)\in\BMO(\rn)\times\BMO(\rn)$, let $T^*_{b,j}$ (j=1,2), $T^*_{\vec{b}}\ $ be the commutators in the j-th entry and the iterated commutators of $T^*$, respectively. 
In this paper, for all $1<p_1,p_2<\infty$, $\frac{1}{p}=\frac{1}{p_1}+\frac{1}{p_2}$, we show that $T^*_{b,j}$ and $T^*_{\vec{b}}$ are compact operators from $L^{p_1}(w_1)\times L^{p_2}(w_2)$ to $L^p(v_{\vec{w}})$, if $b,b_1,b_2\in{\rm CMO}(\mathbb{R}^n)$ and $\vec{w}=(w_1,w_2)\in A_{\vec{p}}$, $v_{\vec{w}}=w_1^{p/p_1}w_2^{p/p_2}$. Here ${\rm CMO}(\mathbb{R}^n)$ denotes the closure of $\mathcal{C}_c^\infty(\mathbb{R}^n)$ in the ${\rm BMO}(\mathbb{R}^n)$ topology and $A_{\vec{p}}$ is the multiple weights class.

\end{abstract}\maketitle

\section{Introduction}

\medskip
The study of commutators of singular intergral operator $T$ can be traced back to the  celebrated works of Coifman, Rochberg and Weiss \cite{MR0412721}. Recall that, for a suitable smooth function $f$, the
commutator of $T$ is defined as
$[b,T]f = T(bf) - bT(f)$. In \cite{MR0412721}, it was shown 
that if $b$ belongs to BMO(${\mathbb{R}^{n}}$), then $[b,T] $ is $ L^p({\mathbb{R}^{n}})$ $(1 < p < \infty)$ bounded. Conversely,
if all commutators of Riesz transform $[R_j , b]$ are $L^p$ bounded for $1 \leq j \leq n$, then  $b \in$ BMO(${\mathbb{R}^{n}}$).

In 1978, Uchiyama \cite{Uch} 
first studied the compactness of commutators of  a singular integral operator 
with the kernel $\Omega\in {\rm Lip}_1({\rm S}^{n-1})$ defined by 
 $$T_{\Omega}f(x)=\textrm{p.v.}\int_{\mathbb{R}^n}\frac{\Omega(y/|y|)}{|y|^n}f(x-y)dy.$$ 
He showed that the commutator $[b,T_\Omega]$ 
is compact on $L^p(\mathbb{R}^n)$ for $1<p<\infty$ if and only if 
$b\in{\rm CMO}(\mathbb{R}^n)$, where ${\rm CMO}(\mathbb{R}^n)$ 
is the closure of $\mathcal{C}_c^\infty(\mathbb{R}^n)$ in the ${\rm BMO}(\mathbb{R}^n)$ 
topology. 

Since then, the study on the compactness of commutators of different operators 
has attracted much more attention. For examples, the compactness of commutators of the linear Fourier multipliers and pseudodifferential
operators was considered by Cordes\cite {Co}. Beatrous and Li \cite{BL} studied the boundedness and compactness of the commutators of Hankel type operators. Krantz and Li \cite{KL1}, \cite{KL2} applied the compactness characterization 
of the commutator $[b,T_\Omega]$ to study Hankel type operators on Bergman space.  Wang\cite{W} showed that the commutators 
of fractional integral operator are compact form $L^p(\rn)$ to $L^q(\rn)$.  In 2009, Chen and Ding 
\cite{CD1} proved thar the commutator of singular integrals with variable kernels is compact 
on $L^p(\rn)$ if and only if $b\in\CMO(\mathbb{R}^n)$ and they also establised the compactness of
Littlewood-Paley square functions in \cite{CD2}. Later on, Chen, Ding and Wang \cite{CDW} 
obtained the compactness of commutators for Marcinkiewicz Integral in Morrey Spaces. 
Recently, Liu, Wang and Xue \cite{LWX} showed the compactness of the commutator of oscillatory singular integrals with non-convolutional type kernels.

In order to state more results, we need to give some definitions and notions. Let $K(x,y,z)$ be a locally integrable function defined
away from the diagonal $x=y=z$ in $\mathbb{R}^{3n}$ and satisfy the conditions as follows:

\begin{enumerate}
	
	\item [\rm{(i)}] \noindent \textbf{The size condition}:  $
	|K(x,y,z)|\leq{A}{(|x-y|+|x-z|)^{-2n}};
	$
	\item [\rm{(ii)}] \noindent \textbf{The regularity conditions}:
	for  $\gamma>0$,  
	$$\indent \quad
	|K(x,y,z)-K_(x',y,z)|\leq \frac{A|x-x'|^{\gamma}}{(|x-y|+|x-z|)^{2n+\gamma}},
	\hbox{\ whenever \ } |x-x'|\leq\frac{1}{2}\max\{|x-y|,|x-z|\};$$
	
	$$\indent \quad
	|K(x,y,z)-K_(x,y',z)|\leq \frac{A|y-y'|^{\gamma}}{(|x-y|+|x-z|)^{2n+\gamma}},
	\hbox{\ whenever \ }|y-y'|\leq\frac{1}{2}\max\{|x-y|,|x-z|\};$$

	$$\indent \quad
	|K(x,y,z)-K_(x,y,z')|\leq \frac{A|z-z'|^{\gamma}}{(|x-y|+|x-z|)^{2n+\gamma}},
	\hbox{\ whenever \ } |z-z'|\leq\frac{1}{2}\max\{|x-y|,|x-z|\}.$$
\end{enumerate}

Let $T$ be the following bilinear singular integral operator
\begin{equation*}
	T(f,g)(x)=\int_{\r2n}K(x,y,z)f(y)g(z)dydz,\quad \hbox{for \ all\ }  x\notin\supp(f)\cap\supp(g),
\end{equation*}
and define the corresponding truncated maximal bilinear singular integral operator $T^*$ by
\begin{equation*}\label{1.3}
	T^*(f,g)(x)=\sup_{\delta>0}\Big|\iint_{|x-y|+|x-z|>\delta}K(x,y,z)f(y)g(z)dydz\Big|.
\end{equation*}
Given $\vec{b}=(b_1,b_2)\in\BMO(\rn)\times\BMO(\rn)$ and $b\in{\rm BMO(\rn)}$, the iterated commutator of $T$ is defined by $$T_{\vec{b}}(f,g)=[b_2,[b_1,T]_1]_2(f,g),$$
and the commutators in the $j$-th entry are defined by
$$[b,T]_1(f,g)=bT(f,g)-T(bf,g) \quad\hbox { and \ \ }[b,T]_2(f,g)=bT(f,g)-T(f,bg).$$
Similarly, the commutators  of $T^*$ in the $j$-th entry are defined by
\begin{equation}\label{1.4}
	T^*_{b,1}(f,g)(x)=\sup_{\delta>0}\Big|\iint_{|x-y|+|x-z|>\delta}K(x,y,z)(b(x)-b(y))f(y)g(z)dydz\Big|,
\end{equation}
\begin{equation}\label{1.5}
	T^*_{b,2}(f,g)(x)=\sup_{\delta>0}\Big|\iint_{|x-y|+|x-z|>\delta}K(x,y,z)(b(x)-b(z))f(y)g(z)dydz\Big|.
\end{equation}
The iterated commutators of $T^*$ is defined by
\begin{align}\label{1.6}
	\begin{split}
		&T^*_{\vec{b}}(f,g)(x)\\
		&=\sup_{\delta>0}\Big|\iint_{|x-y|+|x-z|>\delta}K(x,y,z)(b_1(x)-b_1(y))
		(b_2(x)-b_2(z))f(y)g(z)dydz\Big|.
	\end{split}
\end{align}
It is known that the commutators of bilinear singular integral operator were first defined  and studied by P\'{e}rez and Torres \cite{PT}. In 2009, Lerner et. al. \cite{LOPTT} established the weighted estimates for multilinear operators $T$ and its corresponding iterated commutators $T_{\vec{b}}$ with weight  $\vec{w}\in A_{\vec{p}}$. As for $T^*$, in \cite{C}, Chen gave the weighed $L^p$-estimate for $T^*$ with $A_{\vec{p}}$ weights (see definition \ref{def2.1}). In \cite{X}, Xue proved the following weighted strong type estimate for $T^*_{(b_1,b_2)}$ with $A_{\vec{p}}$ weights.

\quad\hspace{-20pt}{\bf Theorem A} (\cite{X}). {\it  Let  $\vec{w}:=(w_1,w_2)\in A_{\vec{p}}$, $w:=w_1^{p/p_1}w_2^{p/p_2}$, $\frac{1}{p}=\frac{1}{p_1}+\frac{1}{p_2}$ with $1<p_j<\infty$, and $b_j\in\BMO(\rn),\ j=1,2$. Then there is a constant $C>0$ such that
	\begin{equation*}
		\|T^*_{(b_1,b_2)}(f,g)\|_{L^p(w)}\leq C\prod_{j=1}^{2}\|b_j\|_{\BMO}\|f\|_{L^{p_1}(w_1)}\|g\|_{L^{p_2}(w_2)}.
	\end{equation*}
}

\begin{remark}
Using the method of \cite{X}, we can also get the weighted boundedness of $T^*_{b,j},\ j=1,2$.
\begin{equation*}
\|T^*_{b,j}(f,g)\|_{L^p(v_{\vec{w}})}\leq C\|b\|_{\BMO}\|f\|_{L^{p_1}(w_1)}\|g\|_{L^{p_2}(w_2)},
\end{equation*}
with $\vec{w}\in A_{\vec{p}}$, $\frac{1}{p}=\frac{1}{p_1}+\frac{1}{p_2}$, $1<p_j<\infty,\ j=1,2$.
\end{remark}
\smallskip

It is quite natural to consider the compactness of commutators of bilinear operators. B\'enyi and Torres \cite{BT} first showed that  commutators of bilinear Calder\'on-Zygmund operators with multiplication by $\CMO(\rn)$ functions are compact operators from $L^{p_1}\times L^{p_2}$ to $L^p$ for $1<p_1,p_2<\infty$ and $1/p=1/p_1+1/p_2$. Later on, B\'enyi et. al. \cite{BDMT1} gave the following weigthed compactness:
\medskip

\quad\hspace{-20pt}{\bf Theorem B} (\cite{BDMT1}). {\it Suppose $1\leq p_1,p_2<\infty, 1<p<\infty$ with $\frac{1}{p}=\frac{1}{p_1}+\frac{1}{p_2}$, $b\in\CMO(\rn)$, $\vec{b}\in\CMO(\rn)\times\CMO(\rn)$, and $\vec{w}=(w_1,w_2)\in A_p\times A_p$, $v_{\vec{w}}=w_1^{p/p_1}w_2^{p/p_2}$. Then $[b,T]_1$, $[b,T]_2$ and $T_{\vec{b}}$ are compact operators from $L^{p_1}(w_1)\times L^{p_2}(w_2)$ to $L^p(v_{\vec{w}})$.}
\medskip

The compactness for commutators of bilinear fractional integrals were also given in \cite{BDMT}.  
Since then, this subject has attracted much attention. We refer the reader to \cite{DMX,XYY,RC,TX,TXYY} for the compactness of commutators of multilinear operators. Very recently, extrapolation for linear and multilinear compact operators with applications were given by Cao et. al in \cite{COY1, COY2}, and Hyt\"onen et. al. in \cite{H1,H2,H3}. It should be pointed out that all the above results hold for linear not sublinear opertors.

In \cite{BDMT1}, it is pointed out that $A_p\times A_p\subsetneq A_{\vec{p}}$. Hence, a natural question arises:
\begin{question}
Does the compactness result in Theorem B still hold for  multiple weights class $A_{\vec{p}}$?
\end{question}
Recently, Bu and Chen \cite{RC} studied this question and they showed that the commutators of bilinear singular integral operators are compact from $L^{p_1}(w_1)\times L^{p_2}(w_2)$ to $L^p(v_{\vec{w}})$ when $\vec{w}=(w_1,w_2)\in A_{\vec{p}}$, $v_{\vec{w}}\in A_p$ for $1<p<\infty$, $1<p_1,\ p_2<\infty$ with $\frac{1}{p}=\frac{1}{p_1}+\frac{1}{p_2}$. However, they need to assume $v_{\vec{w}}\in A_p$, which leaves this question still open. Moreover, it seems that the condition $p>1$ is not natural, since the strong boundedness holds for $p>1/2.$

Now we turn to the the commutators of bilinear maximal Calderón-Zygmund singular integral operators. In \cite{DMX}, Ding, Mei and Xue obtained the following compactness result:
\medskip

\quad\hspace{-20pt}{\bf Theorem C} (\cite{DMX}). {\it Let $1<p<\infty$, $1<p_1,p_2<\infty$ with $\frac{1}{p}=\frac{1}{p_1}+\frac{1}{p_2}$. If $b\in\CMO(\rn)$, $\vec{b}\in\CMO(\rn)\times\CMO(\rn)$. 
Then the operators $T^*_{b,1}$, $T^*_{b,2}$ and $T^*_{\vec{b}}$ defined by \eqref{1.4}-\eqref{1.6} are compact operators from $L^{p_1}(\rn)\times L^{p_2}(\rn)$ to $L^p(\rn)$.
}
\medskip

Inspired by the above results, one may further ask the following questions:
\begin{question}
Are the operators  $T^*_{b,1}$, $T^*_{b,2}$ and $T^*_{\vec{b}}$ compact from  $L^{p_1}(w_1)\times L^{p_2}(w_2)$ to $L^p(v_{\vec{w}})$ when $\vec{w}$ belongs to $A_{\vec{p}}$ for $1/2<p<\infty.$? 
\end{question}
For $T^*$, the difficulty lies in that, $T^*$ is a sublinear opeator, and the proof in Theorem C relies heavily on the properties of translation invariat operators. However, when it comes to the case of weighted compactness, the weights are not translation invariat. Therefore, these questions become
more difficult and subtle to deal with.

The main purpose of this paper is to give a firm answer to these questions. Smooth truncated techniques will play an essential role in solving these questions.

We summarize our results as follows:

\begin{theorem}\label{thm1.1}
Let $\frac{1}{2}<p<\infty$, $1<p_1,\ p_2<\infty$ with $\frac{1}{p}=\frac{1}{p_1}+\frac{1}{p_2}$  and $\vec{w}=(w_1,w_2)\in A_{\vec{p}}$, $v_{\vec{w}}=w_1^{p/p_1}w_2^{p/p_2}$. If $b\in\CMO(\rn)$,
then the operators $T^*_{b,1}$ and $T^*_{b,2}$ defined by \eqref{1.4} and \eqref{1.5} are compact operators from $L^{p_1}(w_1)\times L^{p_2}(w_2)$ to $L^p(v_{\vec{w}})$.
\end{theorem}
\begin{theorem}\label{thm1.2} Let  $\frac{1}{2}<p<\infty$, $1<p_1,\ p_2<\infty$ with $\frac{1}{p}=\frac{1}{p_1}+\frac{1}{p_2}$  and $\vec{w}=(w_1,w_2)\in A_{\vec{p}}$, $v_{\vec{w}}=w_1^{p/p_1}w_2^{p/p_2}$. If $\vec{b}\in\CMO(\rn)\times\CMO(\rn)$, 
then the operator $T^*_{(b_1,b_2)}$ defined by \eqref{1.6} is a compact operator from $L^{p_1}(w_1)\times L^{p_2}(w_2)$ to $L^p(v_{\vec{w}})$. 
\end{theorem}

Using the general weighted version of Frechet-Kolmogorov theorems (Lemma \ref{lem2.4}), it is easy to see that the condition $v_{\vec{w}}\in A_p$ in Theorem C can be removed. Moreover, the same reasoning as in the proof of Theorem \ref{thm1.1} of this paper shows that the index $1<p<\infty$ can be extended to $1/2<p<\infty$. Therefore. we obtain the following corollary, which gives a firm answer to Question 1.2.

\begin{cor}Let $1/2<p<\infty,$ $1< p_1,p_2<\infty,$ with $\frac{1}{p}=\frac{1}{p_1}+\frac{1}{p_2}$ and $b\in\CMO(\rn)$, $\vec{b}\in\CMO(\rn)\times\CMO(\rn)$, and $\vec{w}=(w_1,w_2)\in A_{\vec{p}}$, $v_{\vec{w}}=w_1^{p/p_1}w_2^{p/p_2}$. Then $[b,T]_1$, $[b,T]_2$ and $T_{\vec{b}}$ are compact operators from $L^{p_1}(w_1)\times L^{p_2}(w_2)$ to $L^p(v_{\vec{w}})$.\end{cor}
\medskip

This paper is organized as follows. In section \ref{S2}, we give some definitions and preliminary lemmas, which are the main ingredients of our proofs. Section \ref{S3} will be devoted to give the proof of Theorem \ref{thm1.1} and Theorem \ref{thm1.2} via smooth truncated techniques. 

Throughout the paper,  the letter $C$ or $c$, sometimes with certain parameters, 
will stand for positive constants not necessarily the same one at each occurrence, 
but are independent of the essential variables. 
\medskip

\section{Definitions and some lemmas }\label{S2}
\medskip
We first recall some notation and definitions. Given
a Lebesgue measurable set $E\subset\mathbb{R}^n$, $|E|$ will denote the Lebesgue measure of $E$. Let $B = B(x, r)$ be a ball in $\mathbb{R}^n$ centered at $x$ with radius $r$ and  $Q(x,r)$ be a cube in $\mathbb{R}^n$ centered at $x$ with the side length $2r$. A weight $w$ is a non-negative measurable and local integrable function on $\mathbb{R}^n$. The measure associated with $w$ is the set function given by $w(E)=\int_E wdx$. For $0<p<\infty$, we denote by $L^p(w)$ the space of all Lebesgue measurable function $f(x)$ such that
$$\|f\|_{L^p(w)}=\Big(\int_{\mathbb{R}^n}|f(x)|^pw(x)dx|\Big)^{1/p}.$$

Recall that a weight $w$ belongs to the classical Muckenhopt class $A_p$$(1<p<\infty)$, if 
\begin{equation*}
\sup_Q\Big(\frac{1}{|Q|}\int_Qw(y)dy\Big)\Big(\frac{1}{|Q|}\int_Qw(y)^{1-p'}dy\Big)^{\frac{p}{p'}}<\infty.
\end{equation*}
$w\in A_1$, if there is a constants $C$ such that
\begin{equation*}
\frac{1}{|Q|}\int_Qw(y)dy\leq C\inf_{x\in Q}w(x),\ \ \ \ for\ a.\ e.\ x\in\rn.
\end{equation*}
The following multiple weights calss $A_{\vec p}$ was first introduced in \cite{LOPTT}.
\begin{definition}{\bf (Multiple $A_{\vec{p}}$ weights class, \cite{LOPTT})}.\label{def2.1}
Let $1\leq p_1,\ p_2<\infty$, $\frac{1}{p}=\frac{1}{p_1}+\frac{1}{p_2}$, $\frac{1}{p_j}+\frac{1}{p_j'}=1,\ j=1,2$. Given $\vec{w}=(w_1,w_2)$, set 
$
v_{\vec{w}}=w_1^{p/p_1}w_2^{p/p_2}.
$
Then, we say $\vec{w}\in A_{\vec{p}}$ if 
\begin{equation*}
[\vec{w}]_{A_{\vec{p}}}=:\sup_Q\Big(\frac{1}{|Q|}\int_Qv_{\vec{w}}\Big)\Big(\frac{1}{|Q|}\int_Qw_1^{1-p_1'}\Big)^{\frac{p}{p_1'}}\Big(\frac{1}{|Q|}\int_Qw_2^{1-p_2'}\Big)^{\frac{p}{p_2'}}<\infty.
\end{equation*}
When $p_j=1,\ j=1,2$, $\Big(\frac{1}{|Q|}\int_Qw_j^{1-p_j'}\Big)^{\frac{p}{p_j'}}$ is understood as $(\inf_Qw_j)^{-1}$.
\end{definition}
It was shown in \cite{LOPTT} that $\vec{w}\in A_{\vec{p}}$ if and only if $v_{\vec{w}}\in A_{2p}$, $w_1^{1-p_1'}\in A_{2p_1'}$ and $w_2^{1-p_2'}\in A_{2p_2'}$.
\medskip

For the classical Muckenhopt weights class $A_p$, it enjoys the following properties:
\begin{lemma}\label{lem2.1}
Let $1<p<\infty$. Then
\begin{enumerate}[{\rm (i)}]
\item If $w\in A_p(\mathbb{R}^n)$, there exists a constant
$\theta\in(0,1)$ such that $w^{1+\theta}\in A_p(\mathbb{R}^n)$. Both
$\theta$ and the $A_p$ constant of $w^{1+\theta}$ depend only on 
$n,\,p$ and the $A_p$ constant of $w$.
\item If $w\in A_p(\mathbb{R}^n)$, we have
\begin{equation*}
\lim\limits_{N\rightarrow+\infty}\int_{|x|>N}\frac{w(x)}{|x|^{np}}dx=0,\ \ \ \lim\limits_{N\rightarrow+\infty}\int_{|x|>N}\frac{w^{1-p'}(x)}{|x|^{np'}}dx=0.
\end{equation*}
\item if $w\in A_\infty=\displaystyle\bigcup_{1\leq p<\infty}A_p$, then there exists a constant $\theta\in(0,1)$ such that for all cubes $Q$ and any set $E\subset Q$, 
$$\frac{w(E)}{w(Q)}\leq C\Big(\frac{|E|}{|Q|}\Big)^{\theta}.$$
\end{enumerate}
\end{lemma}
It should be pointed out that (i) of Lemma \ref{lem2.1} follows
from \cite{CF} and (ii) of Lemma \ref{lem2.1} follows from \cite{GWWY}.

\begin{definition}{\bf (Bilinear Hardy-Littiewood maximal operator, \cite{LOPTT})}.\label{def2.2}
For any $f,\ g\in L^1_{loc}(\rn)$ and $x\in \rn$, the bilinear Hardy-Littiewood maximal operator $\mathcal{M}$ is defined by
$$\mathcal{M}(f,g)(x)=\sup_{Q\ni x}\Big(\frac{1}{Q}\int_Q|f(y)|dy\Big)\Big(\frac{1}{Q}\int_Q|g(z)|dz\Big),$$
where the supremum is taken over all the cubes $Q$ of $\rn$ containing $x$.
\end{definition}
Then, the following characterization of the strong-type inequality holds for $\mathcal{M}$.
\begin{lemma}[\cite{LOPTT}]\label{lem2.2}
Let $1<p_j<\infty,\ j=1,2$ and $\frac{1}{p}=\frac{1}{p_1}+\frac{1}{p_2}$. Let $\vec{w}=(w_1,w_2)\in A_{\vec{p}}$ and $v_{\vec{w}}=w_1^{p/p_1}w_2^{p/p_2}$. Then the inequality
\begin{equation}
\|\mathcal{M}\|_{L^p(v_{\vec{w}})}\leq C\|f\|_{ L^{p_1}(w_1)}\|g\|_{ L^{p_2}(w_2)}
\end{equation}
holds for every $(f,g)\in L^{p_1}(w_1)\times L^{p_2}(w_2)$.
\end{lemma}	

In order to prove Theorem \ref{thm1.1} and Theorem \ref{thm1.2}, we need the following estimates:
\begin{lemma}[\cite{DMX}]\label{lem2.3}
For any $\delta>0$, $0<\epsilon<\frac{1}{2}$, we have the following inequalities
\begin{equation}
\iint_{\frac{\delta}{1+2\epsilon}\leq|y|+|z|\leq\delta}\frac{dydz}{(|y|+|z|)^{2n}}\leq C[1-(1+2\epsilon)^{-n}],
\end{equation}
\begin{equation}\iint_{\delta\leq|y|+|z|\leq\frac{\delta}{1-2\epsilon}}\frac{dydz}{(|y|+|z|)^{2n}}\leq C[(1-2\epsilon-1)^{-n}],
\end{equation}
where the constant $C$ is independent of $\delta$ and $\epsilon$.
\end{lemma}
\medskip
We end this section by introducing the general weighted version of Frechet-Kolmogorov
theorems, which was proved by Xue, Yabuta and Yan in \cite{XYY}.
\begin{lemma}\label{lem2.4}	{\rm (\rm \cite{XYY})}. Let $w$ be a weight on $\mathbb{R}^n$. Assume that $w^{-1/(p_0-1)}$ is also a weight on $\mathbb{R}^n$ for some $p_0>1$. Let $0<p<\infty$ and $\mathcal{F}$ be a subset in $L^p(w)$, then $\mathcal{F}$ is sequentially compact in $L^p(w)$ if the following three conditions are satisfied:
\begin{enumerate}[{\rm (i)}]	
\item $\mathcal{F}$ is bounded, i.e.,
$\sup\limits_{f\in\mathcal{F}}\|f\|_{L^p(w)}<\infty$;
\item $\mathcal{F}$ uniformly vanishes at infinity, i.e.,
\begin{equation*}
\lim\limits_{N\rightarrow\infty}\sup\limits_{f\in\mathcal{F}}\int_{|x|>N}|f(x)|^pw(x)dx=0;
\end{equation*}
\item $\mathcal{F}$ is uniformly equicontinuous, i.e.,
\begin{equation*}
\lim\limits_{|h|\rightarrow0}\sup\limits_{f\in\mathcal{F}}\int_{\mathbb{R}^n}|f(\cdot+h)-f(\cdot)|^pw(x)dx=0.
\end{equation*}
\end{enumerate}
\end{lemma}

\bigskip

\section{Proof of Theorems \ref{thm1.1} and \ref{thm1.2} }\label{S3}
\medskip
In this section, we shall adapt the methods used in \cite{DMX} and \cite{TXYY} to prove Theorem \ref{thm1.1}-\ref{thm1.2}. 

\begin{proof}[Proof of Theorem \ref{thm1.1}]
We only prove $T^*_{b,1}$ is compact, and the proof of the commutator $T^*_{b,2}$ can be get similarly. We shall prove Theorem \ref{thm1.1} via smooth truncated techniques.

First, we introduce the following smooth truncated function. Let $\varphi\in C^{\infty}([0,\infty))$ satisfy
\begin{eqnarray}\label{3.1}
0\leq\varphi\leq1\ \ \ and\ \ \
\varphi(x)=
\begin{cases}
1,          &x\in[0,1],\\
0,          &x\in[2,\infty).
\end{cases}
\end{eqnarray}
For any $\eta>0$, let
\begin{equation}\label{3.2}
K_\eta(x,y,z)=K(x,y,z)\Big[1-\varphi\Big(\frac{2}{\eta}(|x-y|+|x-z|)\Big)\Big].
\end{equation}
Define
\begin{equation}\label{3.3}
T_\eta^*(f,g)(x)=\sup_{\delta>0}\Big|\iint_{|x-y|+|x-z|>\delta}K_\eta(x,y,z)f(y)g(z)dydz\Big|
\end{equation}
and
\begin{equation}\label{3.4}
T_{\eta,b,1}^*(f,g)(x)=\sup_{\delta>0}\Big|\iint_{|x-y|+|x-z|>\delta}K_\eta(x,y,z)(b(x)-b(y))f(y)g(z)dydz\Big|.
\end{equation}

For any $b\in\mathcal{C}_c^\infty(\rn)$ and $\eta>0$, by \eqref{3.2}, \eqref{3.4}, one has
\begin{align*}
\begin{split}
&|T_{b,1}^*(f,g)(x)-T_{\eta,b,1}^*(f,g)(x)|\\
&\leq \sup_{\delta>0}\Big|\iint_{\delta<|x-y|+|x-z|\leq\eta}K(x,y,z)\varphi\Big(\frac{2}{\eta}(|x-y|+|x-z|)\Big)(b(x)-b(y))f(y)g(z)dydz\Big|\\
&\leq \|\nabla b\|_\infty\displaystyle\sup_{\delta>0}\iint_{\delta<|x-y|+|x-z|\leq\eta}\frac{|x-y||f(y)g(z)|}{(|x-y|+|x-z|)^{2n}}dydz\\
&\leq C\eta\displaystyle\sup_{\delta>0}\sum_{j=0}^{\infty}\frac{2^{-j}}{(2^{-j-1}\eta)^{2n}}\iint_{2^{-j-1}\eta<|x-y|+|x-z|\leq2^{-j}\eta}|f(y)g(z)|dydz\\
&\leq C\eta\displaystyle\sup_{\delta>0}\sum_{j=0}^{\infty}\frac{2^{-j}}{(2^{-j}\eta)^{2n}}\int_{Q(x,2^{-j}\eta)}|f(y)|dy\int_{Q(x,2^{-j}\eta)}|g(z)|dz\\
&\leq C\eta \mathcal{M}(f,g)(x).
\end{split}
\end{align*}
Then, using the boundedness of $\mathcal{M}$ from $L^{p_1}(w_1)\times L^{p_2}(w_2)$ to $L^p(v_{\vec{w}})$ (Lemma \ref{lem2.2}), we can obtain
\begin{equation*}
\|T_{b,1}^*(f,g)-T_{\eta,b,1}^*(f,g)\|_{L^p(v_{\vec{w}})}\leq C\eta\|f\|_{L^p(v_{\vec{w}})},
\end{equation*}
which implies that
\begin{equation}\label{3.5}
\lim_{\eta\rightarrow0}\|T_{b,1}^*(f,g)-T_{\eta,b,1}^*(f,g)\|_{L^p(v_{\vec{w}})}=0.
\end{equation}

On the other hand, if $b\in\CMO(\rn)$, then for any $\epsilon>0$, there exists $b_\epsilon\in\mathcal{C}_c^\infty(\rn)$ such that $\|b-b_\epsilon\|_{\BMO(\rn)}<\epsilon$, so that
\begin{equation*}\aligned
\|T_{b,1}^*(f,g)-T_{b_\epsilon,1}^*(f,g)\|_{L^p(v_{\vec{w}})}&\leq\|T_{b-b_\epsilon,1}^*(f,g)\|_{L^p(v_{\vec{w}})}\\&\leq C\|b-b_\epsilon\|_{\BMO(\rn)}\|f\|_{L^{p_1}(w_1)}\|g\|_{L^{p_2}(w_2)}\\&\leq C\epsilon.\endaligned
\end{equation*}
Thus, to prove $T_{b,1}^*$ is compact on 
$L^p(v_{\vec{w}})$ for any $b\in\CMO(\rn)$, it suffices to prove that $T_{b,1}^*$ is compact on $L^p(v_{\vec{w}})$ for any $b\in\mathcal{C}_c^\infty(\mathbb{R}^n)$. By \eqref{3.5} and \cite{Y}, 
it suffices to show that $T^*_{\eta,b,1}$ is compact for any $b\in\mathcal{C}_c^\infty(\mathbb{R}^n)$ when $\eta>0$ is small enough. For arbitrary bounded sets $F\subset L^{p_1}(w_1)$ and $G\subset L^{p_2}(w_2)$, let 
$$\mathcal{F}=\{T_{\eta,b,1}^*(f,g):f\in F,\ g\in G\}.$$
Then, we need to show that for $b\in \mathcal{C}_c^\infty(\mathbb{R}^n)$, $\mathcal{F}$ satisfies the conditions$
\rm (i)$-$\rm(iii)$ of Lemma \ref{lem2.4}. We divide the proof in three steps.

{\bf Step I. } $\mathcal{F}$ satisfies condition
$\rm (i)$ .

First, we verify that $K_\eta(x,y,z)$ satisfies the size and regularity conditions in the introduction section with $\tilde{A}=\max\{2A, 2^{2n+2+\gamma}A\|\nabla\varphi\|_{\infty}\}.$
From the definition of $K_\eta(x,y,z)$, it is easy to see that
\begin{equation}\label{3.6}
|K_\eta(x,y,z)|\leq|K(x,y,z)|\leq\frac{\tilde{A}}{(|x-y|+|x-z|)^{2n}}.
\end{equation}

Next, we will show, when $|x-x'|\leq\frac{1}{2}\max\{|x-y|,
|x-z|\}$, it holds that
\begin{equation}\label{3.7}
|K_{\eta}(x,y,z)-K_{\eta}(x',y,z)|\leq \frac{\tilde{A}|x-x'|^{\gamma}}{(|x-y|+|x-z|)^{2n+\gamma}}
\end{equation}
We consider the following four cases:

(a) $|x-y|+|x-z|\geq\eta$ and $|x'-y|+|x'-z|\geq\eta$. In this case, we have
$K_\eta(x,y,z)=K(x,y,z)$ and $K_\eta(x',y,z)=K(x',y,z)$, which together with the size condition yields \eqref{3.7}.

(b) $|x-y|+|x-z|\geq\eta$ and $|x'-y|+|x'-z|<\eta$. Note that $|x-x'|\leq\frac{1}{2}\max\{|x-y|,|x-z|\}$. In this case, it holds that $K_\eta(x,y)=K(x,y)$ and $\eta>|x'-y|+|x'-z|\geq\frac{1}{2}(|x-y|+|x-z|)$. These together with regularity conditions and $|x-x'|\leq\frac{1}{2}(|x-y|+|x-z|)$ imply that
\begin{align*}
\begin{split}
|K_{\eta}(x,y,z)-K_{\eta}(x',y,z)|&\leq|K(x,y,z)-K(x',y,z)|+|K(x',y,z)|\varphi\Big(\frac{2}{\eta}|x'-y|+|x'-z|\Big)\\
& =|K(x,y,z)-K(x',y,z)|+|K(x',y,z)|\\
&\quad\times\Big|\varphi\Big(\frac{2}{\eta}(|x'-y|+|x'-z|)\Big)-\varphi\Big(\frac{2}{\eta}(|x-y|+|x-z|)\Big)\Big|\\
&\leq\displaystyle\frac{A|x-x'|^{\gamma}}{(|x-y|+|x-z|)^{2n+\gamma}}+
\frac{A|x-x'|}{(|x'-y|+|x'-z|)^{2n}}\frac{4}{\eta}\|\nabla\varphi\|_{\infty}\\
&\leq\displaystyle\frac{A|x-x'|^{\gamma}}{(|x-y|+|x-z|)^{2n+\gamma}}+
\frac{4A\|\nabla\varphi\|_{\infty}|x-x'|}{(|x'-y|+|x'-z|)^{2n+1}}\\
&\leq\displaystyle\frac{A|x-x'|^{\gamma}}{(|x-y|+|x-z|)^{2n+\gamma}}+
\frac{2^{2n+3}A\|\nabla\varphi\|_{\infty}|x-x'|^{\gamma}}{(|x-y|+|x-z|)^{2n+\gamma}}\frac{|x-x'|^{1-\gamma}}{(|x-y|+|x-z|)^{1-\gamma}}\\
&\leq\displaystyle\frac{A|x-x'|^{\gamma}}{(|x-y|+|x-z|)^{2n+\gamma}}+
\frac{2^{2n+3}A\|\nabla\varphi\|_{\infty}|x-x'|^{\gamma}}{(|x-y|+|x-z|)^{2n+\gamma}}(\frac{1}{2})^{1-\gamma}\\
&\leq \displaystyle\frac{\tilde{A}|x-x'|^{\gamma}}{(|x-y|+|x-z|)^{2n+\gamma}}.
\end{split}
\end{align*}
which proves \eqref{3.7}.

(c) $|x-y|+|x-z|<\eta$ and $|x'-y|+|x'-z|\geq\eta$. This case is similar to
case (b).

(d) $|x-y|+|x-z|<\eta$ and $|x'-y|+|x'-z|<\eta$. Then, we have
\begin{align*}
\begin{split}
|K_{\eta}(x,y,z)-K_{\eta}(x',y,z)|&\leq|K(x,y,z)-K(x',y,z)|+|K(x,y,z)-K(x',y,z)|\varphi\Big(\frac{2}{\eta}(|x-y|+|x-z|)\Big)\\
&\quad+|K(x',y,z)|\Big|\varphi\Big(\frac{2}{\eta}(|x'-y|+|x'-z|)\Big)-\varphi\Big(\frac{2}{\eta}(|x-y|+|x-z|)\Big)\Big|\\
&\leq\displaystyle\frac{2A|x-x'|^{\gamma}}{(|x-y|+|x-z|)^{2n+\gamma}}+
\frac{4A\|\nabla\varphi\|_{\infty}|x-x'|^{\gamma}}{(|x-y|+|x-z|)^{2n+\gamma}}(\frac{1}{2})^{1-\gamma}\\
&\leq \displaystyle\frac{\tilde{A}|x-x'|^{\gamma}}{(|x-y|+|x-z|)^{2n+\gamma}}.
\end{split}
\end{align*}

Similar arguments to those in deriving \eqref{3.7} may give that
\begin{equation}\label{3.8}
|K(x,y,z)-K(x,y',z)|\leq\frac{\tilde{A}|y-y'|^\gamma}{(|x-y|+|x-z|)^{n+\gamma}},
\end{equation}
whenever $|y-y'|\leq\frac{1}{2}\max\{|x-y|,|x-z|\}$, and
\begin{equation}\label{3.9}
|K_\eta(x,y,z)-K_\eta(x,y,z')|\leq\frac{\tilde{A}|z-z'|^\gamma}{(|x-y|+|x-z|)^{n+\gamma}},
\end{equation}
whenever $|z-z'|\leq\frac{1}{2}\max\{|x-y|,|x-z|\}$.

Hence, the weighted strong type estimate also hold for $T_\eta^*$, $T^*_{\eta,{b,j}},\ j=1,2$.  Thus, we have
\begin{equation*}
\sup\limits_{f\in F, g\in G}\|T^*_{\eta,b,1}(f,g)\|_{L^p(v_{\vec{w}})}\leq C\sup\limits_{f\in F,g\in G}\|f\|_{L^{p_1}(w_1)}\|g\|_{L^{p_2}(w_2)}\leq C<\infty,
\end{equation*}
which yields the fact that the set $\mathcal{F}$ is bounded.

{\bf Step II.} $\mathcal{F}$ satisfies condition
$\rm (ii)
$.

We adapt the method using in \cite{RC} to verify the condition (ii) of Lemma \ref{lem2.4}. Assume $b\in\mathcal{C}_c^\infty(\mathbb{R}^n)$ and $\supp(b)\subset B(0,R)$, where $B(0,R)$ is the ball of radius $R>1$ center at origin in $\mathbb{R}^n$. For any $|x| > N > 2R$ and $l>0$, 
denote
$$T_0(f,g)(x)=\int_{|z|<|x|}\int_{|y|<R}|k_\eta(x,y,z)||f(y)||g(z)|dydz,$$
and for $l\geq1$,
$$T_l(f,g)(x)=\int_{2^{l-1}|x|<|z|<2^l|x|}\int_{|y|<R}|k_\eta(x,y,z)||f(y)||g(z)|dydz.$$
Since $\vec{w}\in A_{\vec{p}}$, we have $w_1^{-p_1'/p_1}=w_1^{1-p_1'}\in A_{2p_1'}\subset A_\infty(\rn)$, then, by Lemma \ref{lem2.1} (iii), there exists a constant $\theta\in(0,1)$ such that
$$
\int_{B(0,R)}w_1^{-p_1'/p_1}(y)dy\leq C(2^{-(j+l)}R/N)^{n\theta}\int_{B(0,2^{l+j}N)}w_1^{-p_1'/p_1}(y)dy
$$
holds  for any integers $l\geq0,j\geq1$. This gives that
\begin{align*}
\begin{split}
T_l(f,g)(x)&\leq C\int_{2^{l-1}|x|<|z|<2^l|x|}\int_{|y|<R}\frac{|f(y)||g(z)|}{(|x|+|x-z|)^{2n}}dydz\\
&\leq C(2^{l-1}|x|)^{-2n}\int_{2^{l-1}|x|<|z|<2^l|x|}\int_{|y|<R}|f(y)||g(z)|dydz\\
&\leq C(2^{l-1}|x|)^{-2n}\|f\|_{L^{p_1}(w_1)}^{p_1}\|g\|_{L^{p_2}(w_2)}^{p_2}\Big(\int_{B(0,R)}w_1^{-p_1'/p_1}(y)dy\Big)^{1/p_1'}\\
&\quad\times\Big(\int_{B(0,2^{l}|x|)}w_1^{-p_2'/p_2}(y)dy\Big)^{1/p_2'}\\
&\leq C(2^{l-1}|x|)^{-2n}(2^{-(j+l)}R/N)^{n\theta/p_1'}\Big(\int_{B(0,2^{l+j}N)}w_1^{-p_1'/p_1}(y)dy\Big)^{1/p_1'}\\
&\quad\times\Big(\int_{B(0,2^{l}|x|)}w_1^{-p_2'/p_2}(y)dy\Big)^{1/p_2'}.
\end{split}
\end{align*}
In the same way, we can get
$$
T_0(f,g)(x)\leq C|x|^{-2n}(2^{-j}R/N)^{n\theta/p_1'}\Big(\int_{B(0,2^{j}N)}w_1^{-p_1'/p_1}(y)dy\Big)^{1/p_1'}
\Big(\int_{B(0,|x|)}w_1^{-p_2'/p_2}(y)dy\Big)^{1/p_2'}.
$$
Therefore, for $p\geq1$, by Minkowski's inequality, we have
\begin{align*}
\begin{split}
&\Big(\int_{|x|>N}|T^*_{\eta,b,1}(f,g)(x)|^pv_{\vec{w}}(x)dx\Big)^{1/p}\\
&\leq C\sum_{j=1}^{\infty}\Big(\int_{2^{j-1}N<|x|<2^{j}N}\Big(\int_{\rn}\int_{|y|<R}|k_\eta(x,y,z)||f(y)||g(z)|dydz\Big)^pv_{\vec{w}}(x)dx\Big)^{1/p}\\
&\leq C\sum_{j=1}^{\infty}\sum_{l=0}^{\infty}\Big(\int_{2^{j-i}N<|x|<2^jN}|T_l(f,g)(x)|^pv_{\vec{w}}(x)dx\Big)^{1/p}\\
&\leq C\sum_{j=1}^{\infty}\sum_{l=0}^{\infty}(2^{l+j-2}N)^{-2n}(2^{-(j+l)}R/N)^{n\theta /p_1'}\Big(\int_{B(0,2^jN)}v_{\vec{w}}(x)dx\Big)^{1/p}\\
&\quad\times\Big(\int_{B(0,2^{j+l}N)}w_1^{-p_1'/p_1}(y)dy\Big)^{1/p_1'}\Big(\int_{B(0,2^{j+l}N)}w_1^{-p_2'/p_2}(y)dy\Big)^{1/p_2'}\\
&\leq C\sum_{j=1}^{\infty}2^{-jn\theta /p_1'}\sum_{l=0}^{\infty}2^{-ln\theta /p_1'}(R/N)^{n\theta /p_1'}\leq C(R/N)^{n\theta /p_1'}.
\end{split}
\end{align*}
When $p<1$, since $(\sum_{l=0}^{\infty}a_l)^p\leq\sum_{l=0}^{\infty}a_l^p$, similar to the above estimate, we can obtain
\begin{align*}
\begin{split}
\int_{|x|>N}|T^*_{\eta,b,1}(f,g)(x)|^pv_{\vec{w}}(x)dx
&\leq C\sum_{j=1}^{\infty}\int_{2^{j-1}N<|x|<2^{j}N}\Big|\sum_{l=0}^{\infty}T_l(f,g)(x)\Big|^pv_{\vec{w}}(x)dx\\
&\leq C\sum_{j=1}^{\infty}\sum_{l=0}^{\infty}\int_{2^{j-i}N<|x|<2^jN}|T_l(f,g)(x)|^pv_{\vec{w}}(x)dx\\
&\leq C(R/N)^{n\theta p/p_1'}.
\end{split}
\end{align*}
Then, for any $1/2<p<\infty$, we have
\begin{equation*}
\lim_{N\rightarrow\infty}\int_{|x|>N}|T^*_{\eta,b,1}(f,g)(x)|^pv_{\vec{w}}(x)dx=0
\end{equation*}
holds whenever $f\in F, g\in G$.

{\bf Step III.} $\mathcal{F}$ satisfies condition
$\rm (iii)
$.

It remains to show that the set $\mathcal{F}$ is uniformly equicontinuous.
It suffices to verify that for any $0<\epsilon<1/4$, if $|h|$ is sufficiently small and dependents only on $\epsilon$, then 
\begin{equation}\label{3.10}
\|T^*_{\eta,b,1}(f,g)(h+\cdot)-T^*_{\eta,b,1}(f,g)(\cdot)\|_{L^p(v_{\vec{w}})}\leq C\epsilon,
\end{equation}
holds uniformly for $f\in F,g\in G$.

In what follows, we fix $\eta\in(0,\frac{1}{16})$ and $|h|<\frac{\epsilon\eta}{4}$. Denote $K_{\eta,\delta}(x,y,z)=k_\eta(x,y,z)\chi_{\{|x-y|+|x-z|>\delta\}}$. Then, we have the following decomposition
\begin{align}\label{3.11}
\begin{split}
&|T^*_{\eta,b,1}(f,g)(x+h)-T^*_{\eta,b,1}(f,g)(x)|\\
&\leq\displaystyle \sup_{\delta>0}\Big|\int_{\r2n}K_{\eta,\delta}(x+h,y,z)(b(x+h)-b(y))f(y)g(z)dydz\\
&\quad-\int_{\r2n}K_{\eta,\delta}(x,y,z)(b(x)-b(y))f(y)g(z)dydz\Big|\\
&\leq\displaystyle \sup_{\delta>0}\Big|\int_{\r2n}K_{\eta,\delta}(x,y,z)(b(x+h)-b(x))f(y)g(z)dydz\Big|\\
&\quad+\displaystyle\sup_{\delta>0}\Big|\int_{\r2n}(K_{\eta,\delta}(x+h,y,z)-K_{\eta,\delta}(x,y,z))(b(x+h)-b(y))f(y)g(z)dydz\Big|\\
&=:I_1(x)+I_2(x).
\end{split}
\end{align}
	
For $I_1(x)$, it holds that
\begin{equation*}
\begin{array}{ll}
&I_1(x)=\displaystyle|b(x+h)-b(x)|\sup_{\delta>0}\Big|\iint_{|x-y|+|x-z|>\delta}K_{\eta}(x,y,z)f(y)g(z)dydz\Big|\\
&\qquad\leq C\displaystyle |h|T^*_{\eta}(f,g)(x)
\end{array}
\end{equation*}
Then, by the boundedness of $T^*_{\eta}(f,g)$, we have
\begin{equation}\label{3.12}
\|I_1\|_{L^p(v_{\vec{w}})}\leq C|h|\|f\|_{L^{p_1}(w_1)}\|g\|_{L^{p_2}(w_2)}\leq C|h|\leq C\epsilon.
\end{equation}
	
Next we will estimate $I_2(x)$. When $|x-y|+|x-z|<\frac{\eta}{4}$ and $|h|<\frac{\epsilon\eta}{4}<\frac{\eta}{16}$, then $|x+h-y|+|x+h-z|<\frac{3\eta}{8}$.
Hence
$$\varphi(\frac{2}{\eta}(|x+h-y|+|x+h-z|))=1=\varphi(\frac{2}{\eta}(|x-y|+|x-z|)).$$
This implies
$$K_\eta(x+h,y,z)=0=K_\eta(x,y,z).$$
Then, for $I_{2}(x)$, we decompose it as follows:
\begin{align}\label{3.13}
\begin{split}
I_2(x)&\leq\displaystyle \sup_{\delta>0}\Big|\iint_{|x-y|+|x-z|>\frac{\eta}{4}}(K_{\eta}(x+h,y,z)-K_{\eta}(x,y,z))\chi_{\{|x+h-y|+|x+h-z|>\delta\}}\\
&\quad\times(b(x+h)-b(y))f(y)g(z)dydz\Big|\\
&\quad+\sup_{\delta>0}\Big|\iint_{|x-y|+|x-z|>\frac{\eta}{4}}K_{\eta}(x,y,z)(\chi_{\{|x+h-y|+|x+h-z|>\delta\}}-\chi_{\{|x-y|+|x-z|>\delta\}})\\
&\quad\times(b(x+h)-b(y))f(y)g(z)dydz\Big|\\
&=:I_{21}(x)+I_{22}(x).
\end{split}
\end{align}

For $I_{21}(x)$, When $|x-y|+|x-z|\geq\frac{\eta}{4}$ and $|h|<\frac{\epsilon\eta}{4}<\frac{\eta}{16}$, we have $|h|\leq\frac{1}{2}\max\{|x-y|,|x-z|\}$. then, using \eqref{3.7} and splitting into annuli, we obtain
\begin{equation}\label{3.14}
\begin{split}
I_{21}(x)&\leq C\displaystyle \sup_{\delta>0}\iint_{|x-y|+|x-z|>\frac{\eta}{4}}\frac{|h|^\gamma}{(|x-y|+|x-z|)^{2n+\gamma}}|f(y)g(z)|dydz\\
&\leq C|h|^\gamma\eta^{-\gamma}\displaystyle \sup_{\delta>0}\frac{2^{-j\gamma}}{(\frac{\eta}{4}2^j)^{2n}}\sum_{j=0}^{\infty}\iint_{|x-y|+|x-z|\sim\frac{\eta}{4}2^j}|f(y)g(z)|dydz\\
&\leq C|h|^\gamma\eta^{-\gamma}\displaystyle \sup_{\delta>0}\frac{2^{-j\gamma}}{(\frac{\eta}{4}2^j)^{2n}}\sum_{j=0}^{\infty}\int_{Q(x,\frac{\eta}{4}2^j)}|f(y)|dy\int_{Q(x,\frac{\eta}{4}2^j)}|g(z)|dz\\
&\leq C|h|^\gamma\eta^{-\gamma}\mathcal{M}(f,g)(x).
\end{split}
\end{equation}
Then, for any $(f,g)\in{L^{p_1}(w_1)}\times{L^{p_2}(w_2)}$, by the boundedness of $\mathcal{M}$, we have
\begin{equation*}
\|I_{21}\|_{L^p(v_{\vec{w}})}\leq C|h|^\gamma\eta^{-\gamma}\|\mathcal{M}(f,g)\|_{L^p(v_{\vec{w}})}\leq C|h|^\gamma\eta^{-\gamma}\|f\|_{L^{p_1}(w_1)}\|g\|_{L^{p_2}(w_2)}.
\end{equation*}
Taking $\gamma=|h|^{1/2}$, we get
\begin{equation}\label{3.15}
\|I_{21}\|_{L^p(v_{\vec{w}})}\leq C|h|^{\gamma/2}\leq C\epsilon.
\end{equation}
For $I_{22}(x)$, we have
\begin{align}\label{3.16}
\begin{split}
I_{22}(x)&\leq \sup_{\delta>0}\Big|\iint_{{|x-y|+|x-z|>\frac{\eta}{4}\atop|x-y|+|x-z|\leq\delta}\atop
|x+h-y|+|x+h-z|>\delta}K_{\eta}(x,y,z)(b(x+h)-b(y))f(y)g(z)dydz\Big|\\
&\quad+ \sup_{\delta>0}\Big|\iint_{{|x-y|+|x-z|>\frac{\eta}{4}\atop|x-y|+|x-z|>\delta}\atop
|x+h-y|+|x+h-z|\leq\delta}K_{\eta}(x,y,z)(b(x+h)-b(y))f(y)g(z)dydz\Big|\\
&=:I_{22}^1(x)+I_{22}^2(x).
\end{split}
\end{align}

For$I_{22}^1(x)$, since $|h|<\frac{\epsilon\eta}{4}, 0<\epsilon<\frac{1}{4}$ and $|x-y|+|x-z|>\frac{\eta}{4},\ |x-y|+|x-z|\leq\delta,\ |x+h-y|+|x+h-z|>\delta$, then $|x-y|+|x-z|>\epsilon^{-1}|h|$ and $|x-y|+|x-z|\geq\frac{\delta}{1+2\epsilon}$. Therefore, for any $1<r<\min\{p_1,p_2\}$, by \eqref{3.7}, lemma \ref{lem2.3} and H\"older inequality, we get

\begin{align*}
\begin{split}
I_{22}^1(x)&\leq C\|b\|_\infty \sup_{\delta>0}\iint_{\frac{\delta}{1+2\epsilon}\leq|x-y|+|x-z|\leq\delta}\frac{|f(y)g(z)|}{(|x-y|+|x-z|)^{2n}}dydz\\
&\leq C \sup_{\delta>0}\Big(\iint_{\frac{\delta}{1+2\epsilon}\leq|x-y|+|x-z|\leq\delta}\frac{|f(y)g(z)|^r}{(|x-y|+|x-z|)^{2n}}dydz\Big)^{\frac{1}{r}}\\
&\quad\times\Big(\iint_{\frac{\delta}{1+2\epsilon}\leq|y|+|z|\leq\delta}\frac{dydz}{(|y|+|z|)^{2n}}dydz\Big)^{\frac{1}{r'}}\\
&\leq C \sup_{\delta>0}\Big((1+2\epsilon)^{2n}\delta^{-2n}\iint_{\frac{\delta}{1+2\epsilon}\leq|x-y|+|x-z|\leq\delta}|f(y)g(z)|^rdydz\Big)^{\frac{1}{r}}[1-(1+2\epsilon)^{-n}]^{\frac{1}{r'}}\\
&\leq C(1+2\epsilon)^{2n}[1-(1+2\epsilon)^{-n}]^{\frac{1}{r'}}\sup_{\delta>0}\Big(\delta^{-2n}\int_{Q(x,\delta)}|f(y)|^rdy\int_{Q(x,\delta)}|g(z)|^rdz\Big)^{\frac{1}{r}}\\
&\leq C\epsilon^{1/r'}[\mathcal{M}(|f|^r,|g|^r)(x)]^{\frac{1}{r}}.
\end{split}
\end{align*}
Notice that $(f, g)\in L^{p_1}(w_1)\times L^{p_2}(w_2)$, $\frac{1}{p}=\frac{1}{p_1}+\frac{1}{p_2}$, then $(|f|^r,|g|^r)\in L^{p_1/r}(w_1)\times L^{p_2/r}(w_2)$, $\frac{r}{p}=\frac{r}{p_1}+\frac{r}{p_2}$. Thus,  the boundedness of $\mathcal{M}$ gives that

\begin{align}\label{3.17}
\begin{split}
\|I_{22}^1\|_{L^p(v_{\vec{w}})}&\leq C\epsilon^{1/r'}\Big(\int_{\rn}|\mathcal{M}(|f|^r,|g|^r)(x)|^{\frac{p}{r}}v_{\vec{w}}(x)dx\Big)^{\frac{1}{p}}\\
&=C\epsilon^{1/r'}\|\mathcal{M}(|f|^r,|g|^r)\|_{L^{p/r}(v_{\vec{w}})}^{1/r}\\
&\leq C\epsilon^{1/r'}\||f|^r\|_{L^{p_1/r}(w_1)}^{1/r}\||g|^r\|_{L^{p_2/r}(w_2)}^{1/r}\\
&=C\epsilon^{1/r'}\|f\|_{L^{p_1}(w_1)}\|g\|_{L^{p_2}(w_2)}\\
&\leq C\epsilon.
\end{split}
\end{align}

For$I_{22}^2(x)$, since $|h|<\frac{\epsilon\eta}{4}, 0<\epsilon<\frac{1}{4}$ and $|x-y|+|x-z|>\frac{\eta}{4},\ |x-y|+|x-z|>\delta,\ |x+h-y|+|x+h-z|\leq\delta$, then $|x-y|+|x-z|>\epsilon^{-1}|h|$ and $|x-y|+|x-z|\leq\frac{\delta}{1-2\epsilon}$. Then,for any $1<r<\min\{p_1,p_2\}$, by H\"older inequality, \eqref{3.7} and lemma \ref{lem2.3}, we get
\begin{align*}
\begin{split}
I_{22}^2(x)&\leq C\|b\|_\infty \sup_{\delta>0}\iint_{\delta\leq|x-y|+|x-z|\leq\frac{\delta}{1-2\epsilon}}\frac{|f(y)g(z)|}{(|x-y|+|x-z|)^{2n}}dydz\\
&\leq C \sup_{\delta>0}\Big(\iint_{\delta\leq|x-y|+|x-z|\leq\frac{\delta}{1-2\epsilon}}\frac{|f(y)g(z)|^r}{(|x-y|+|x-z|)^{2n}}dydz\Big)^{\frac{1}{r}}\\
&\quad\times\Big(\iint_{\delta\leq|y|+|z|\leq\frac{\delta}{1-2\epsilon}}\frac{dydz}{(|y|+|z|)^{2n}}dydz\Big)^{\frac{1}{r'}}\\
&\leq C \sup_{\delta>0}\Big(\delta^{-2n}\iint_{\delta\leq|x-y|+|x-z|\leq\frac{\delta}{1-2\epsilon}}|f(y)g(z)|^rdydz\Big)^{\frac{1}{r}}[(1-2\epsilon)^{-n}-1]^{\frac{1}{r'}}\\
&\leq C\epsilon^{1/r'}[\mathcal{M}(|f|^r,|g|^r)(x)]^{\frac{1}{r}}.
\end{split}
\end{align*}
Similar to \eqref{3.17}, it holds that
\begin{equation}\label{3.18}
\|I_{22}^2\|_{L^p(v_{\vec{w}})}\leq C\epsilon,
\end{equation}
which together with \eqref{3.12}, \eqref{3.15} and \eqref{3.17} yields \eqref{3.10}. This finishes the proof of Theorem \ref{thm1.1}.
\end{proof}

\begin{proof}[Proof of Theorem \ref{thm1.2}]We still use smooth truncated techniques to prove Theorem \ref{thm1.2}. Similarly as in the proof of Theorem \ref{thm1.1}, we define
\begin{equation*}
T_{\eta,\vec{b}}^*(f,g)(x)=\sup_{\delta>0}\Big|\iint_{|x-y|+|x-z|>\delta}K_\eta(x,y,z)(b_1(x)-b_1(y))(b_2(x)-b_2(y))f(y)g(z)dydz\Big|.
\end{equation*}	 Then, for any $\vec{b}\in\mathcal{C}_c^\infty(\rn)\times\mathcal{C}_c^\infty(\rn)$, it is easy to get that
\begin{align*}
\begin{split}
&|T_{\vec{b}}^*(f,g)(x)-T_{\eta,\vec{b}}^*(f,g)(x)|\\
&\leq \|\nabla b_1\|_\infty\|\nabla b_2\|_\infty\displaystyle\sup_{\delta>0}\iint_{\delta<|x-y|+|x-z|\leq\eta}\frac{|x-y||x-z||f(y)g(z)|}{(|x-y|+|x-z|)^{2n}}dydz\\
&\leq C\eta^2\displaystyle\sup_{\delta>0}\sum_{j=0}^{\infty}\frac{2^{-2j}}{(2^{-j-1}\eta)^{2n}}\iint_{2^{-j-1}\eta<|x-y|+|x-z|\leq2^{-j}\eta}|f(y)g(z)|dydz\\
&\leq C\eta^2\displaystyle\sup_{\delta>0}\sum_{j=0}^{\infty}\frac{2^{-2j}}{(2^{-j}\eta)^{2n}}\int_{Q(x,2^{-j}\eta)}|f(y)|dy\int_{Q(x,2^{-j}\eta)}|g(z)|dz\\
&\leq C\eta^2\mathcal{ M}(f,g)(x).
\end{split}
\end{align*}
Therefore, using the boundedness of $\mathcal{ M}$ from $L^{p_1}(w_1)\times L^{p_2}(w_2)$ to $L^p(v_{\vec{w}})$, we can obtain
\begin{equation*}
\|T_{\vec{b}}^*(f,g)-T_{\eta,\vec{b}}^*(f,g)\|_{L^p(v_{\vec{w}})}\leq C\eta^2\|f\|_{L^p(w)},
\end{equation*}
which implies that
\begin{equation*}
\lim_{\eta\rightarrow0}\|T_{\vec{b}}^*(f,g)-T_{\eta,\vec{b}}^*(f,g)\|_{L^p(w)}=0.
\end{equation*}

On the other hand, if $\vec{b}\in\CMO(\rn)\times\CMO(\rn)$, then for any $\epsilon>0$, there exists $\vec{b}^\epsilon=(b_1^\epsilon,b_2^\epsilon)\in\mathcal{C}_c^\infty(\rn)\times\mathcal{C}_c^\infty(\rn)$ such that $\|b_j-b_j^\epsilon\|_{\BMO(\rn)}<\epsilon,\ j=1,2$, so that
\begin{align*}
\begin{split}
\|T_{\vec{b}}^*(f,g)-T_{\vec{b}^\epsilon}^*(f,g)\|_{L^p(v_{\vec{w}})}
&\leq\|T_{(b_1-b_1^\epsilon,b_2)}^*(f,g)\|_{L^p(v_{\vec{w}})}+\|T_{(b_1,b_2-b_2^\epsilon)}^*(f,g)\|_{L^p(v_{\vec{w}})}\\
&\leq C\|b_2\|_{\BMO(\rn)}\|b_1-b_1^\epsilon\|_{\BMO(\rn)}\|f\|_{L^{p_1}(w_1)}\|g\|_{L^{p_2}(w_2)}\\
&\quad+C\|b_1\|_{\BMO(\rn)}\|b_2-b_2^\epsilon\|_{\BMO(\rn)}\|f\|_{L^{p_1}(w_1)}\|g\|_{L^{p_2}(w_2)}\\
&\leq C\epsilon.
\end{split}
\end{align*}
Thus, to prove $T_{\vec{b}}^*$ is compact on 
$L^p(v_{\vec{w}})$ for any $\vec{b}\in\CMO(\rn)\times\CMO(\rn)$, we only need to show that $T_{\eta,\vec{b}}^*$ is compact for any $\vec{b}\in\mathcal{C}_c^\infty(\mathbb{R}^n)\times\mathcal{C}_c^\infty(\mathbb{R}^n)$ when $\eta>0$ is small enough. For arbitrary bounded sets $F\subset L^{p_1}(w_1)$ and $G\subset L^{p_2}(w_2)$, let 
$$\mathcal{G}=\{T_{\eta,\vec{b}}^*(f,g):f\in F,\ g\in G\}.$$
Then, we  shall prove that for any $\vec{b}\in\mathcal{C}_c^\infty(\mathbb{R}^n)\times\mathcal{C}_c^\infty(\mathbb{R}^n)$, $\mathcal{G}$ satisfies the conditions$\rm (i)$-$\rm(iii)$ of Lemma \ref{lem2.4}.

Similar to the proof of Theorem \ref{thm1.1}, we can get \rm(i) holds easily. 

Assume $b_j\in\mathcal{C}_c^\infty(\mathbb{R}^n)$ and $\supp(b_j)\subset B(0,R)$,\ $j=1,2$ , where $B(0,R)$ is the ball of radius $R>1$ center at origin in $\mathbb{R}^n$. For any $|x| > N > 2R$, $\vec{w}\in A_{\vec{p}}$, by H\"older inequality, we have
\begin{align*}
\begin{split}
T_{\eta,\vec{b}}^*(f,g)(x)&\leq C\int_{B(0,R)}\int_{B(0,R)}\frac{|b_1(y)|b_2(z)|}{(|x-y|+|x-z|)^{2n}}|f(y)||g(z)|dydz\\
&\leq C\|b_1\|_\infty\|b_2\|_\infty|x|^{-2n}\int_{B(0,R)}\int_{B(0,R)}|f(y)||g(z)|dydz\\
&\leq C|x|^{-2n}\|f\|_{L^{p_1}(w_1)}\|g\|_{L^{p_2}(w_2)}\Big(\int_{B(0,R)}w_1^{-p_1'/p_1}(y)dy\Big)^{\frac{1}{p_1'}}\Big(\int_{B(0,R)}w_1^{-p_2'/p_2}(y)dy\Big)^{\frac{1}{p_2'}}.
\end{split}
\end{align*}
 Thus, it follows that
\begin{equation*}
\int_{|x|>N}|T_{\eta,\vec{b}}^*(f,g)(x)|^pv_{\vec{w}}(x)dx\leq C\int_{|x|>N}\frac{v_{\vec{w}}(x)}{|x|^{2np}}dx,
\end{equation*}
Since $\vec{w}\in A_{\vec{p}}$, $\frac{1}{2}<p<\infty$, then $1<2p<\infty$ and $v_{\vec{w}}\in A_{2p}(\rn)$ (see \cite{LOPTT}), which together with Lemma \ref{lem2.1} (ii) yields that
\begin{equation*}
\lim\limits_{A\rightarrow+\infty}\int_{|x|>A}|T_{\eta,\vec{b}}^*(f,g)(x)|^pv_{\vec{w}}(x)dx=0,
\end{equation*}
 whenever $f\in F$ and $g\in G$.

It remains to show that the set $\mathcal{F}$ is uniformly equicontinuous.
It suffices to verify that for any $0<\epsilon<1/4$, if $|h|$ is sufficiently small and dependent only on $\epsilon$, then 
\begin{equation}\label{3.19}
\|T_{\eta,\vec{b}}^*(f,g)(h+\cdot)-T_{\eta,\vec{b}}^*(f,g)(\cdot)\|_{L^p(v_{\vec{w}})}\leq C\epsilon,
\end{equation}
holds uniformly for $f\in F,g\in G$.

 Fix $\eta\in(0,\frac{1}{16})$ and $|h|<\frac{\epsilon\eta}{4}$. Denote $K_{\eta,\delta}(x,y,z)=K_\eta(x,y,z)\chi_{\{|x-y|+|x-z|>\delta\}}$ and $a(x,y,z)=(b_1(x)-b_1(y))(b_2(x)-b_2(y))$, $a_h(x,y,z)=a(x+h,y,z)-a(x,y,z)$. Then
\begin{align}\label{3.20}
\begin{split}
&|T_{\eta,\vec{b}}^*(f,g)(x+h)-T_{\eta,\vec{b}}^*(f,g)(x)|\\
&\leq\displaystyle \sup_{\delta>0}\Big|\int_{\r2n}K_{\eta,\delta}(x+h,y,z)a(x+h,y,z)f(y)g(z)dydz\\
&\quad-\int_{\r2n}K_{\eta,\delta}(x,y,z)a(x,y,z)f(y)g(z)dydz\Big|\\
&\leq\displaystyle \sup_{\delta>0}\Big|\int_{\r2n}K_{\eta,\delta}(x,y,z)a_h(x,y,z)f(y)g(z)dydz\Big|\\
&\quad+\displaystyle\sup_{\delta>0}\Big|\int_{\r2n}(K_{\eta,\delta}(x+h,y,z)-K_{\eta,\delta}(x,y,z))a(x+h,y,z)f(y)g(z)dydz\Big|\\
&=:J_1(x)+J_2(x).
\end{split}
\end{align}

For $J_1(x)$, notice that
\begin{align*}
\begin{split}
a_h(x,y,z)=&(b_1(x+h)-b_1(x))(b_2(x+h)-b_2(x))+(b_1(x+h)-b_1(x))(b_2(x)-b_2(z))\\
&+(b_1(x)-b_1(y))(b_2(x+h)-b_2(x)).
\end{split}
\end{align*}
Then, we obtain
\begin{align*}
\begin{split}
J_1(x)&\leq\displaystyle|h|^2\|\nabla b_1\|_\infty\|\nabla b_2\|_\infty\sup_{\delta>0}\Big|\iint_{|x-y|+|x-z|>\delta}K_{\eta}(x,y,z)f(y)g(z)dydz\Big|\\
&\quad+|h|\|\nabla b_1\|_\infty\sup_{\delta>0}\Big|\iint_{|x-y|+|x-z|>\delta}K_{\eta}(x,y,z)(b_2(x)-b_2(z))f(y)g(z)dydz\Big|\\
&\quad+|h|\|\nabla b_2\|_\infty\sup_{\delta>0}\Big|\iint_{|x-y|+|x-z|>\delta}K_{\eta}(x,y,z)(b_1(x)-b_1(y))f(y)g(z)dydz\Big|\\
&\leq C|h|^2T_\eta^*(f,g)(x)+C|h|T_{\eta,b_2,2}^*(f,g)(x)+C|h|T_{\eta,b_1,1}^*(f,g)(x).
\end{split}
\end{align*}
By the weighted boundedness of $T^*_{\eta}(f,g)$ and $T_{\eta,b_j,j}^*,\ j=1,2$, we have
\begin{equation}\label{3.21}
\|J_1\|_{L^p(v_{\vec{w}})}\leq C(|h|^2+|h|)\|f\|_{L^{p_1}(w_1)}\|g\|_{L^{p_2}(w_2)}.
\end{equation}

For $J_2(x)$. When $|x-y|+|x-z|<\frac{\eta}{4}$ and $|h|<\frac{\epsilon\eta}{4}<\frac{\eta}{16}$, then $|x+h-y|+|x+h-z|<\frac{3\eta}{8}$.
Hence
$\varphi(\frac{2}{\eta}(|x+h-y|+|x+h-z|))=1=\varphi(\frac{2}{\eta}(|x-y|+|x-z|)).$
This implies$$K_\eta(x+h,y,z)=K_\eta(x,y,z)=0.$$ Then, we decompose it as follows:
\begin{align}\label{3.22}
\begin{split}
J_2(x)&\leq\displaystyle \sup_{\delta>0}\Big|\iint_{|x-y|+|x-z|>\frac{\eta}{4}}(K_{\eta}(x+h,y,z)-K_{\eta}(x,y,z))\chi_{\{|x+h-y|+|x+h-z|>\delta\}}\\
&\quad\times a(x+h,y,z)f(y)g(z)dydz\Big|\\
&\quad+\sup_{\delta>0}\Big|\iint_{|x-y|+|x-z|>\frac{\eta}{4}}K_{\eta}(x,y,z)(\chi_{\{|x+h-y|+|x+h-z|>\delta\}}-\chi_{\{|x-y|+|x-z|>\delta\}})\\
&\quad\times a(x+h,y,z)f(y)g(z)dydz\Big|\\
&=:J_{21}(x)+J_{22}(x).
\end{split}
\end{align}

	For $J_{21}(x)$, when $|x-y|+|x-z|\geq\frac{\eta}{4}$ and $|h|<\frac{\epsilon\eta}{4}<\frac{\eta}{16}$, we have $|h|\leq\frac{1}{2}\max\{|x-y|,|x-z|\}$. Then, using \eqref{3.7} and splitting into annuli, we obtain
\begin{equation*}
\begin{split}
J_{21}(x)&\leq C\displaystyle \sup_{\delta>0}\iint_{|x-y|+|x-z|>\frac{\eta}{4}}\frac{|h|^\gamma}{(|x-y|+|x-z|)^{2n+\gamma}}|f(y)g(z)|dydz\\
&\leq C|h|^\gamma\eta^{-\gamma}\displaystyle \sup_{\delta>0}\frac{2^{-j\gamma}}{(\frac{\eta}{4}2^j)^{2n}}\sum_{j=0}^{\infty}\iint_{|x-y|+|x-z|\sim\frac{\eta}{4}2^j}|f(y)g(z)|dydz\\
&\leq C|h|^\gamma\eta^{-\gamma}\mathcal{M}(f,g)(x).
\end{split}
\end{equation*}
Then, for any $(f,g)\in{L^{p_1}(w_1)}\times{L^{p_2}(w_2)}$, by the boundedness of $\mathcal{M}$, we have
\begin{equation*}
\|J_{21}\|_{L^p(v_{\vec{w}})}\leq C|h|^\gamma\eta^{-\gamma}\|\mathcal{M}(f,g)\|_{L^p(v_{\vec{w}})}\leq C|h|^\gamma\eta^{-\gamma}\|f\|_{L^{p_1}(w_1)}\|g\|_{L^{p_2}(w_2)}.
\end{equation*}
Taking $\gamma=|h|^{1/2}$, we get
\begin{equation}\label{3.23}
\|J_{21}\|_{L^p(v_{\vec{w}})}\leq C|h|^{\gamma/2}\leq C\epsilon.
\end{equation}

For $J_{22}(x)$, we have
\begin{align*}
\begin{split}
I_{22}(x)&\leq \sup_{\delta>0}\Big|\iint_{{|x-y|+|x-z|>\frac{\eta}{4}\atop|x-y|+|x-z|\leq\delta}\atop
	|x+h-y|+|x+h-z|>\delta}K_{\eta}(x,y,z)a(x+h,y,z)f(y)g(z)dydz\Big|\\
&\quad+ \sup_{\delta>0}\Big|\iint_{{|x-y|+|x-z|>\frac{\eta}{4}\atop|x-y|+|x-z|>\delta}\atop
	|x+h-y|+|x+h-z|\leq\delta}K_{\eta}(x,y,z)a(x+h,y,z)f(y)g(z)dydz\Big|\\
&=:J_{22}^1(x)+J_{22}^2(x).
\end{split}
\end{align*}
Analogous to the estimates of $I_{22}^1(x)$ and $I_{22}^2(x)$, we can obtain
\begin{align*}
\begin{split}
J_{22}^1(x)&\leq C\|b_1\|_\infty\|b_2\|_\infty \sup_{\delta>0}\iint_{\frac{\delta}{1+2\epsilon}\leq|x-y|+|x-z|\leq\delta}\frac{|f(y)g(z)|}{(|x-y|+|x-z|)^{2n}}dydz\\
&\leq C \sup_{\delta>0}\Big((1+2\epsilon)^{2n}\delta^{-2n}\iint_{\frac{\delta}{1+2\epsilon}\leq|x-y|+|x-z|\leq\delta}|f(y)g(z)|^rdydz\Big)^{\frac{1}{r}}[1-(1+2\epsilon)^{-n}]^{\frac{1}{r'}}\\
&\leq C\epsilon^{1/r'}[\mathcal{M}(|f|^r,|g|^r)(x)]^{\frac{1}{r}}.
\end{split}
\end{align*}
and
\begin{align*}
\begin{split}
J_{22}^2(x)&\leq C\|b_1\|_\infty\|b_2\|_\infty \sup_{\delta>0}\iint_{\delta\leq|x-y|+|x-z|\leq\frac{\delta}{1-2\epsilon}}\frac{|f(y)g(z)|}{(|x-y|+|x-z|)^{2n}}dydz\\
&\leq C \sup_{\delta>0}\Big(\delta^{-2n}\iint_{\delta\leq|x-y|+|x-z|\leq\frac{\delta}{1-2\epsilon}}|f(y)g(z)|^rdydz\Big)^{\frac{1}{r}}[(1-2\epsilon)^{-n}-1]^{\frac{1}{r'}}\\
&\leq C\epsilon^{1/r'}[\mathcal{M}(|f|^r,|g|^r)(x)]^{\frac{1}{r}}.
\end{split}
\end{align*}
Therefore, we have
\begin{equation}
\|J_{22}^1\|_{L^p(v_{\vec{w}})}\leq C\epsilon,\ \ \ \ \ \ \ |J_{22}^2\|_{L^p(v_{\vec{w}})}\leq C\epsilon,
\end{equation}
which together with \eqref{3.21}  and \eqref{3.23} yields \eqref{3.19} and completes the proof of Theorem \ref{thm1.2}.

\end{proof}

\end{document}